\documentclass[12pt]{amsart}

\usepackage{amscd,amssymb,amsmath,graphicx,verbatim}
\usepackage{wasysym}
\usepackage{url}
\usepackage{doi}
\usepackage{hyperref}
\usepackage[TS1,OT1,T1]{fontenc}
\usepackage{lscape} %landscape
\usepackage{fullpage} %full page
\usepackage{pslatex} %times new roman
\usepackage[all]{xy}

\newtheorem{theorem}{Theorem}[section]
\newtheorem{lemma}[theorem]{Lemma}
\newtheorem{corollary}[theorem]{Corollary}
\newtheorem{proposition}[theorem]{Proposition}

\theoremstyle{definition}
\newtheorem{definition}[theorem]{Definition}

\newtheorem{example}[theorem]{Example}

\theoremstyle{remark}

\usepackage{xcolor}% provides \colorlet
\usepackage{fixme}
\fxsetup{
    status=draft,
    author=,
    layout=margin,
    theme=color
}
\definecolor{fxnote}{rgb}{1.0000,0.0000,0.0000}
\colorlet{fxnotebg}{yellow}

\makeatletter
\renewcommand*\FXLayoutInline[3]{%
  \@fxdocolon {#3}{%
    \@fxuseface {inline}%
    \colorbox{fx#1bg}{\color {fx#1}\ignorespaces #3\@fxcolon #2}}}
\makeatother
\DeclareMathOperator{\Ext}{Ext}

\newcommand{\Dcal}{\ensuremath{\mathcal{D}}}
\newcommand{\Xcal}{\ensuremath{\mathcal{X}}}

\newcommand{\Tcal}{\ensuremath{\mathcal{T}}}
\newcommand{\Fcal}{\ensuremath{\mathcal{F}}}
\newcommand{\Ccal}{\ensuremath{\mathcal{C}}}

\newcommand{\Wcal}{\ensuremath{\mathcal{W}}}

\newcommand{\Kbb}{\mathbb{K}}

\newcommand{\ra}{\rightarrow}

\newcommand{\tr}[2]{t_{#1}(#2)}

\numberwithin{equation}{section}

\begin{document}
\title{Torsion classes, wide subcategories and localisations}
\author{Frederik Marks, Jan {\v S}{\v{t}}ov{\'{\i}}{\v{c}}ek}

\address{Frederik Marks, University of Stuttgart, Institute for Algebra and Number Theory, Pfaffenwaldring 57, 70569 Stuttgart, Germany}
\email{marks@mathematik.uni-stuttgart.de}

\address{Jan {\v S}{\v{t}}ov{\'{\i}}{\v{c}}ek, Charles University in Prague, Faculty of Mathematics and Physics, Department of Algebra, Sokolovsk\'a 83, 186 75 Praha, Czech Republic}
\email{stovicek@karlin.mff.cuni.cz}

\thanks{The second named author was supported by grant GA~\v{C}R 14-15479S from the Czech Science Foundation.}

\begin{abstract}
For a finite dimensional algebra $A$, we establish correspondences between torsion classes and wide subcategories in $mod(A)$. In case $A$ is representation finite, we obtain an explicit bijection between these two classes of subcategories. Moreover, we translate our results to the language of ring epimorphisms and universal localisations. It turns out that universal localisations over representation finite algebras are classified by torsion classes and support $\tau$-tilting modules.
\end{abstract}
\maketitle

% ================================================================================
\section{Introduction}
Torsion classes and wide (i.e abelian and extension closed) subcategories of a given module category play an important role in the representation theory of rings and algebras. 
In this article, we show that these two classes of subcategories are intrinsically connected. If $A$ is a finite dimensional algebra, we obtain an injection from wide subcategories to torsion classes in $mod(A)$. When restricting to functorially finite subcategories this map can be turned around yielding an injection from functorially finite torsion classes to functorially finite wide subcategories. Consequently, over representation finite algebras torsion classes are in bijection with wide subcategories.

We use this bijection to classify all universal localisations in the sense of \cite{Sch}. These localisations turn out to correspond to torsion classes and also to support $\tau$-tilting modules recently introduced in \cite{AIR}. This connection not only allows us to readily learn basic parameters of the corresponding localisations (such as the number of simple modules over the localised ring), but also opens the possibility to study the still somewhat mysterious universal localisations using advanced tools of representation theory. For instance, \cite{AIR} introduces a well behaved mutation procedure for support $\tau$-tilting modules and it might be interesting to understand its counterpart for universal localisations.

\smallskip

Looking back at the concepts involved, the notion of a torsion class in an abelian category goes back to Dickson (\cite{D}). These classes occur as the left-hand-side of a pair of subcategories $(\Tcal,\Fcal)$ where $\Tcal$ and $\Fcal$ are chosen maximally with respect to the property that $Hom(T,F)=0$ for all $T\in\Tcal$ and $F\in\Fcal$. When passing to module categories, torsion classes are studied, for example, in tilting theory (see \cite{AHK}). In fact, the subcategory of all modules generated by a tilting module forms a torsion class. Moreover, torsion classes are relevant to decide if a given subcategory of a module category has almost split sequences (\cite{AS2}). Building on this work, it was shown in \cite{AIR} that functorially finite torsion classes in $mod(A)$ are parametrised by support $\tau$-tilting modules.

Wide subcategories of a given module category are full abelian subcategories that are closed under extensions. In \cite{H}, Hovey classified the wide subcategories of finitely presented modules over some commutative rings by certain subsets of the spectrum of the ring. This classification can be understood as a module theoretical interpretation of the correspondence between thick subcategories of perfect complexes in the derived category and certain unions of closed sets of the spectrum (see \cite{Hop,N,T}).  
Passing to non-commutative rings, it was shown in \cite{IT} that over a finite dimensional hereditary algebra $A$, the functorially finite wide subcategories of $mod(A)$ are in bijection with the functorially finite torsion classes of the same category.

Universal localisations, as introduced by Cohn and Schofield (see \cite{Sch}), generalise classical localisation theory to possibly non-commutative rings. Instead of inverting a given set of elements, one localises with respect to a set of maps between finitely generated projective modules. Although universal localisations were already successfully used in different areas of mathematics like topology (\cite{R}), algebraic K-theory (\cite{NR}) and tilting theory (\cite{AA}), the concept itself is not yet fully understood. In our context, universal localisations turn out to be relevant, since the modules over the localised ring always form a wide subcategory of the initial module category.

\smallskip

In this article, we work over an arbitrary finite dimensional algebra $A$ over a field and we are interested in the category $mod(A)$ of finite dimensional $A$-modules. We start by associating a torsion class to every wide subcategory and we show that this assignment is injective (Proposition \ref{prop wide}). However, in general, there are more torsion classes than wide subcategories (Example \ref{ex kro}). In a second step, we restrict the setting to functorially finite subcategories and we obtain an injection in the other direction -- from functorially finite torsion classes to functorially finite wide subcategories (Proposition \ref{prop alpha injective}). We also give sufficient conditions for this map to be bijective (Theorem \ref{main} and Corollary \ref{cor main}). In a final section, we discuss applications of our results to ring epimorphisms and universal localisations. In fact, classifying functorially finite wide subcategories in $mod(A)$ amounts to classifying certain epimorphisms of rings starting in $A$. We provide sufficient conditions for these ring epimorphisms to be universal localisations (Proposition \ref{prop univ-loc-Tor}). Moreover, we show that for representation finite algebras torsion classes and wide subcategories are in bijection with universal localisations (see Theorem \ref{main 2}).

% =================================================================================
\section{Notation}
Throughout, let $A$ be a finite dimensional algebra over an algebraically closed field $\Kbb$. The category of all finitely generated left $A$-modules will be denoted by $mod(A)$. All subcategories $\Ccal$ of $mod(A)$ are considered to be full and replete. A subcategory $\Ccal$ of $mod(A)$ is called \emph{functorially finite}, if every $A$-module admits both a left and a right $\Ccal$-approximation. More precisely, for every $X\in mod(A)$ we need objects $C_1,C_2\in\Ccal$ and morphisms $g_1:X\ra C_1$ and $g_2:C_2\ra X$ such that the maps $Hom_A(g_1,\tilde{C})$ and $Hom_A(\tilde{C},g_2)$ are surjective for all $\tilde{C}\in\Ccal$. Moreover, a subcategory $\Ccal$ is called \emph{wide}, if it is exact abelian and closed for extensions and it is called a \emph{torsion class}, if it is closed for quotients and extensions. By $wide(A)$ (respectively, $f\mbox{-}wide(A)$) we denote the class of all (functorially finite) wide subcategories in $mod(A)$. Moreover, by $tors(A)$ (respectively, $f\mbox{-}tors(A)$) we denote the class of all (functorially finite) torsion classes in $mod(A)$. For a module $X$ in $mod(A)$, we denote by $add(X)$ the subcategory of $mod(A)$ containing all direct summands of finite direct sums of copies of $X$. For a given subcategory $\Ccal$, $gen(\Ccal)$ describes the subcategory containing all quotients of finite direct sums of objects from $\Ccal$ and the subcategory $filt(\Ccal)$ is built from all finitely generated $A$-modules $X$ admitting a finite filtration of the form
$$0=F_0\subseteq F_1\subseteq F_2\subseteq ... \subseteq F_n= X$$
with $F_i/F_{i-1}\in\Ccal$ for all $1\leq i\leq n$ and $n\in\mathbb{N}$. We say that the module $X$ possesses a $\Ccal$-\emph{filtration of length} $n$ and the minimal length of such a filtration is called the $\Ccal$-\emph{length} of $X$. Finally, note that $filt(\Ccal)$ is the smallest subcategory of $mod(A)$ containing $\Ccal$ and being closed under extensions.

% ==============================================================================
\section{Torsion classes and wide subcategories}

We start by associating a torsion class to every wide subcategory $\Wcal$ in $wide(A)$. We set $$\Tcal_\Wcal:=filt(gen(\Wcal)).$$

\begin{lemma}\label{lem torsion}
$\Tcal_\Wcal$ is a torsion class in $mod(A)$.
\end{lemma}

\begin{proof}
By construction, it suffices to check closure for quotients. Take $X$ in $\Tcal_\Wcal$ and a surjection $X\ra X'$ in $mod(A)$. We show by induction on the $gen(\Wcal)$-length of $X$ that also $X'$ belongs to $\Tcal_\Wcal$. Clearly, if $X$ lies in $gen(\Wcal)\subseteq\Tcal_\Wcal$, then so does $X'$. For the general case, take a short exact sequence of the form
$$\xymatrix{0\ar[r] & M_1\ar[r] & X\ar[r] & M_2\ar[r] & 0}$$
with $M_1$ in $gen(\Wcal)$ and $M_2$ in $\Tcal_\Wcal$ of strictly smaller $gen(\Wcal)$-length than $X$. We get the following induced commutative diagram
$$\xymatrix{0\ar[r] & M_1\ar@{=}[d]\ar[r] & X\ar[d]\ar[r] & M_2\ar[d]\ar[r] & 0\\ & M_1\ar[r]^f & X'\ar[r] & coker(f)\ar[r] & 0}$$
By induction hypothesis, $coker(f)\in\Tcal_\Wcal$ and, thus, $X'\in\Tcal_\Wcal$, as an extension of modules in $\Tcal_\Wcal$.
\end{proof}

\begin{lemma}\label{lem subobjects}
The category $\Wcal$ is closed for subobjects in $\Tcal_\Wcal$. More precisely, if $W \in \Wcal$ and $X \subseteq W$ is a submodule that belongs to $\Tcal_\Wcal$, then $X \in \Wcal$. 
\end{lemma}

\begin{proof}
Take an object $W$ in $\Wcal$ together with an injection $i:X\ra W$ where $X$ belongs to $\Tcal_\Wcal$. We have to show that $X\in\Wcal$. We proceed by induction on the $gen(\Wcal)$-length of $X$. To begin with, if $X$ is in $gen(\Wcal)$, then it already lies in $\Wcal$ - as the image of a map in $\Wcal$. For the general case, take a short exact sequence of the form
$$\xymatrix{0\ar[r] & M_1\ar[r] & X\ar[r] & M_2\ar[r] & 0}$$
as in the proof of Lemma \ref{lem torsion}. We get the following induced commutative diagram
$$\xymatrix{0\ar[r] & M_1\ar[r]^{i_1}\ar[d] & W\ar[r]\ar@{=}[d] & coker(i_1)\ar[r]\ar[d]^\pi & 0\\ 0\ar[r] & X\ar[r]^i & W\ar[r] & coker(i)\ar[r] & 0}$$
Since $M_1$ lies in $gen(\Wcal)$ it already belongs to $\Wcal$ and, thus, so does $coker(i_1)$. Moreover, by the Snake lemma, we have $M_2\cong ker(\pi)$ and, hence, by induction hypothesis, also $M_2$ is in $\Wcal$. Finally, since $\Wcal$ is closed under extensions, one gets $X\in\Wcal$ - as wanted.
\end{proof}

Next, we show how to recover the wide subcategory $\Wcal$ from the torsion class $\Tcal_\Wcal$. We need the following construction due to \cite{IT}. Let $\Tcal$ be  a torsion class. We define 
$$\alpha(\Tcal):=\{X\in\Tcal\mid\forall(g:Y\ra X)\in\Tcal, ker(g)\in\Tcal\}.$$

Following \cite[Proposition 2.12]{IT}, $\alpha(\Tcal)$ is a wide subcategory of $mod(A)$. This result is actually an instance of a more general fact~\cite[Exercise 8.23]{KS}, since $\alpha(\Tcal)$ consists precisely of what is called \emph{$\Tcal$-coherent objects} in~\cite{KS}.

\begin{proposition}\label{prop wide}
Let $\Wcal$ be a wide subcategory of $mod(A)$. Then $\Wcal=\alpha(\Tcal_\Wcal)$. In particular, mapping $\Wcal$ to $\Tcal_\Wcal$ yields an injection from $wide(A)$ to $tors(A)$.
\end{proposition}

\begin{proof}
We start by showing that $\Wcal$ is contained in $\alpha(\Tcal_\Wcal)$. Take $W$ in $\Wcal$ and a test map $g:Y\ra W$ with $Y$ in $\Tcal_\Wcal$. We have to show that $ker(g)$ belongs to $\Tcal_\Wcal$. First, assume that $Y$ belongs to $gen(\Wcal)$. Thus, there is some $W'$ in $\Wcal$ and a surjection $\pi$ yielding the following commutative diagram of $A$-modules
$$\xymatrix{0\ar[r] & ker(g')\ar[r]\ar[d]_{\pi'} & W'\ar[r]^{g'}\ar[d]_\pi & W\ar@{=}[d]\\ 0\ar[r] & ker(g)\ar[r] & Y\ar[r]^g & W}$$
It follows that also the map $\pi'$ must be surjective.
Therefore, since $ker(g')$ belongs to $\Wcal$, the kernel of $g$ is in $gen(\Wcal)\subseteq\Tcal_\Wcal$ - as wanted. For the general case, take a short exact sequence
$$\xymatrix{0\ar[r] & M_1\ar[r]^i & Y\ar[r] & M_2\ar[r] & 0}$$
with $M_1$ in $gen(\Wcal)$ and $M_2$ in $\Tcal_\Wcal$ of strictly smaller $gen(\Wcal)$-length than $Y$. We get the following induced commutative diagram of $A$-modules with exact columns
$$\xymatrix{0\ar[r] & ker(g\circ i)\ar[r]\ar[d]^{i_1} & M_1\ar[r]\ar[d]^{i} & Im(g\circ i)\ar[d]\ar[r] & 0\\ 0\ar[r] & ker(g)\ar[d]\ar[r] & Y\ar[d]\ar[r]^g & W\ar[d] & \\ 0\ar[r] & coker(i_1)\ar[r] & M_2\ar[r] & coker(g\circ i) & }$$
By Lemma \ref{lem subobjects}, $Im(g\circ i)$ is in $\Wcal$ and, thus, so is $coker(g\circ i)$. Now it follows from induction that $ker(g\circ i)$ and $coker(i_1)$ belong to $\Tcal_\Wcal$. Hence, since $\Tcal_\Wcal$ is closed for extensions, $ker(g)$ is in $\Tcal_\Wcal$.

Next, we check that $\alpha(\Tcal_\Wcal)\subseteq\Wcal$. Take a module $X$ in $\alpha(\Tcal_\Wcal)$. First, assume that $X$ belongs to $gen(\Wcal)$. Thus, there is some $W\in\Wcal$ and a surjection $\pi:W\ra X$. By assumption, $K:=ker(\pi)$ lies in $\Tcal_\Wcal$. Hence, by Lemma~\ref{lem subobjects}, $K$ is in $\Wcal$ and, since $\Wcal$ is wide, $X \in \Wcal$.
For the general case, take a short exact sequence of the form
$$\xymatrix{0\ar[r] & M_1\ar[r] & X\ar[r] & M_2\ar[r] & 0}$$
with $M_1$ in $gen(\Wcal)$ and $M_2$ in $\Tcal_\Wcal$ of strictly smaller $gen(\Wcal)$-length than $X$. Since, by definition, $\alpha(\Tcal_\Wcal)$ is closed for subobjects in $\Tcal_\Wcal$, it follows that also $M_1$ belongs to $\alpha(\Tcal_\Wcal)$ and, hence, so does $M_2$. Using induction, this already implies that $M_1$ and $M_2$ are in $\Wcal$. In particular, $X$ belongs to $\Wcal$ - as an extension of $M_1$ and $M_2$. This finishes the proof.
\end{proof}

The following example illustrates that, in general, for a given algebra $A$ there are more torsion classes than wide subcategories in $mod(A)$.

\begin{example}\label{ex kro}
Let $A$ be the Kronecker algebra defined as the path algebra over $\mathbb{K}$ of the quiver
$$\xymatrix{\bullet\ar@<.5ex>[r]\ar@<-.5ex>[r] &\bullet}$$
Let $\Tcal\subseteq mod(A)$ be the torsion class given by all preinjective $A$-modules. It is not hard to check that $\alpha(\Tcal)=\{0\}$. In particular, there is no $\Wcal\in wide(A)$ such that $\Tcal_\Wcal=\Tcal$.
\end{example}

Next, we restrict the setting to functorially finite torsion classes and functorially finite wide subcategories. First we recall important related definitions.

\begin{definition}[{\cite[Definition 3.10]{AMV}}] \label{defn silting}
Let $T \in mod(A)$ and $\sigma: P_1 \to P_0$ be a projective presentation of $T$. Let us denote by $\Dcal_\sigma$ the class
\[ \Dcal_\sigma := \{X\in mod(A)\mid Hom_A(\sigma,X)\text{ surjective}\}. \]
We say that $T$ is \emph{silting with respect to $\sigma$} if $\Dcal_\sigma=gen(T)$.
We say that $T$ is \emph{silting} if there exists a projective presentation $\sigma$ of $T$ with respect to which $T$ is silting.
\end{definition}

This concept was defined in \cite{AMV} in the category of all modules, but if $T \in mod(A)$, our definition is easily seen to be equivalent. It was shown in~\cite[Proposition 3.16]{AMV} that finite dimensional silting modules over $A$ coincide with support $\tau$-tilting modules from~\cite{AIR}.

Now \cite[Theorem 2.7]{AIR} asserts that the functorially finite torsion classes $\Tcal \subseteq mod(A)$ are parametrised by the isomorphism classes of basic silting modules, i.e.\ the silting modules $T$ with a decomposition $T=M_1\oplus\cdots\oplus M_n$ where all the $M_i$ are indecomposable and $M_i\not\cong M_j$ for all $i\not= j$. Given $T$ silting, the corresponding functorially finite torsion class is simply $\Tcal=gen(T)$. If we conversely start with a torsion class $\Tcal$, there is a $\Tcal$-approximation sequence
$$\xymatrix@1{A\ar[r]^\phi & T_0\ar[r] & T_1\ar[r] & 0}$$
with $\phi$ left-minimal, i.e.\ for all $h\in End_A(T_0)$, $h\circ\phi=\phi$ implies that $h$ is an isomorphism. 
Now, the corresponding silting module is given by the direct sum of all non-isomorphic indecomposable direct summands of $T_0$ and $T_1$ (compare Lemma~\ref{lem split-proj} below).

\begin{definition} \label{defn split-proj}
Let $\Ccal$ be a class of $A$-modules. A module $P \in \Ccal$ is \emph{Ext-projective} in $\Ccal$ if we have $\Ext^1_A(P,M) = 0$ for all $M \in \Ccal$. The module $P$ is \emph{split projective} in $\Ccal$ if each surjective morphism $M \to P$ with $M \in \Ccal$ splits.
\end{definition}

\begin{lemma} \label{lem split-proj}
Let $\Tcal \subseteq mod(A)$ be a functorially finite torsion class and let
\[ \xymatrix@1{A\ar[r]^{\phi} & T_0\ar[r] & T_1\ar[r] & 0} \]
be an exact sequence such that $\phi$ is a minimal left $\Tcal$-approximation. Then $add(T_0 \oplus T_1)$ is the class of all Ext-projective modules in $\Tcal$. Moreover, an indecomposable module $P \in add(T_0 \oplus T_1)$ is split projective in $\Tcal$ if and only if it is a summand of $T_0$.
\end{lemma}

\begin{proof}
We need to prove only the last sentence, the rest follows essentially from~\cite[\S1]{AIR}. Suppose that $P$ is an indecomposable summand of $T_0 \oplus T_1$. 
If $P$ is split projective in $\Tcal$, then, since $\Tcal = gen(T_0)$, there is a split epimorphism $T_0^n \to P$ and $P$ is a summand of $T_0$.

Suppose conversely that $P$ is a summand of $T_0$ and consider a surjection $p\colon X \to P$ with $X \in \Tcal$. We must prove that $p$ splits. Since $\Tcal = gen(T_0)$, we can without loss of generality assume that $X \in add(T_0)$. Moreover, $p$ splits if and only if $p \oplus 1_{T_0/P}\colon X \oplus T_0/P \to T_0$ splits. Hence, it suffices to prove that every surjection $p'\colon X' \to T_0$ with $X' \in add(T_0)$ splits. To this end, $\phi\colon A \to T_0$ factors as $\phi = p'\circ f$ for some $f\colon A \to X'$, since $p'$ is surjective. On the other hand, since $X' \in add(T_0)$, there is a factorization $f = i\circ\phi$ for some $i\colon T_0 \to X'$ yielding the commutative diagram
\[
\xymatrix{
& A \ar[dl]_\phi \ar[d]^f \ar[dr]^\phi \\
T_0 \ar[r]_i & X' \ar@{->>}[r]_{p'} & T_0
}
\]
Putting this all together, we have $\phi = p'\circ i\circ\phi$. Since $\phi$ is left-minimal, $p'$ splits as required.
\end{proof}

For an $A$-module $M$, we denote by $M^\circ$ the subcategory of $mod(A)$ given by all modules $X$ such that $Hom_A(M,X)=0$. For a functorially finite torsion class $\Tcal=gen(T)$ in $mod(A)$ and $T_1$ chosen as above, we are interested in the intersection $\Wcal_\Tcal:=\Tcal\cap T_1^\circ$. Note that, since $T_1$ is Ext-projective in $\Tcal$, it follows that $(gen(T_1),T_1^\circ)$ is a torsion pair in $mod(A)$. In particular, using the main theorem in \cite{S}, $T_1^\circ$ is a functorially finite subcategory of $mod(A)$. Consequently, also $\Wcal_\Tcal$ is functorially finite. Indeed, we can construct left $\Wcal_\Tcal$-approximations by first taking a left $\Tcal$-approximation and then factoring out the $gen(T_1)$-torsion part; right approximations are constructed dually (compare to the approach in \cite{AMV2}).

\begin{lemma}\label{lemma alternative description}
Let $\Tcal=gen(T)$ be a functorially finite torsion class in $mod(A)$. Then $\alpha(\Tcal)$ coincides with $\Wcal_\Tcal$. In particular, $\alpha(\Tcal)$ is functorially finite.
\end{lemma}

\begin{proof}
First, take $X$ in $\Wcal_\Tcal$ and consider a map $g:Y\ra X$ in $\Tcal$. We have to check that $ker(g)\in\Tcal$. Note that, by assumption, also $Im(g)$ lies in $\Wcal_\Tcal$.
Now we choose a surjection $\pi:A^n\ra ker(g)$ for some $n>0$. Using the left $\Tcal$-approximation $\phi:A\ra T_0$, we obtain the following commutative diagram of $A$-modules
$$\xymatrix{& A^n\ar[r]^{\phi^n}\ar[d]_{\pi} & T_0^n\ar[r]\ar[d] & T_1^n\ar[d]\ar[r] & 0\\ 0\ar[r] & ker(g)\ar[r] & Y\ar[r]^g & Im(g)\ar[r] & 0}$$
Since $Hom_A(T_1,Im(g))=0$, there is a surjection $\tilde{\pi}:T_0^n\ra ker(g)$ such that $\pi=\tilde{\pi}\circ\phi^n$. Hence, $ker(g)$  belongs to $\Tcal$.

For the other inclusion, we take some $X$ in $\alpha(\Tcal)$. It suffices to show that $Hom_A(T_1,X)=0$. Let $h:T_1\ra X$ be such a map. We will prove that $Im(h)$ is zero. First of all, observe that also $Im(h)$ lies in $\alpha(\Tcal)$. Now consider the following induced commutative diagram of $A$-modules
$$\xymatrix{& A\ar@{=}[r]\ar[d]^{\bar{\phi}} & A\ar[d]^\phi & & \\ 0\ar[r] & ker(\bar{h})\ar[r]^{i_0}\ar[d] & T_0\ar[r]^{\bar{h}}\ar[d] & Im(h)\ar[r]\ar@{=}[d] & 0\\ 0\ar[r] & ker(h)\ar[r]\ar[d] & T_1\ar[r]^h\ar[d] & Im(h)\ar[r] & 0\\ & 0 & 0 & &}$$
Since $Im(h)$ belongs to $\alpha(\Tcal)$, the module $ker(\bar{h})$ has to be in $\Tcal$. Hence, using that $\phi$ is a left $\Tcal$-approximation, there is a map $s:T_0\ra ker(\bar{h})$ with $s\circ\phi=\bar{\phi}$. It follows that $\phi=i_0\circ\bar{\phi}=i_0\circ s\circ\phi$. Since $\phi$ is left-minimal, $i_0\circ s$ has to be an isomorphism and, thus, so does $i_0$. In particular, $Im(h)$ is zero, as wanted.
\end{proof}

Now we are able to state the following key proposition in our context.

\begin{proposition}\label{prop alpha injective}
Let $\Tcal$ be in $f\mbox{-}tors(A)$. Then $\Tcal=\Tcal_{\alpha(\Tcal)}=filt(gen(\alpha(\Tcal)))$. In particular, mapping $\Tcal$ to $\alpha(\Tcal)$ yields an injection from $f\mbox{-}tors(A)$ to $f\mbox{-}wide(A)$.
\end{proposition}

\begin{proof}
Since, by construction, $\Tcal_{\alpha(\Tcal)}$ is the smallest torsion class containing $\alpha(\Tcal)$, it is enough to check that $\Tcal\subseteq\Tcal_{\alpha(\Tcal)}$. We take a non-zero module $X$ in $\Tcal$. By induction, it suffices to show that $Hom_A(\alpha(\Tcal),X)\not= 0$. Suppose that the opposite holds. Let $\phi\colon A\ra T_0$ be the minimal left $\Tcal$-approximation of the regular module. In what follows, we view $Hom_A(T_0,X)$ as a right $End_A(T_0)$-module. Note that $Hom_A(T_0,X)\not= 0$, since $X$ is in $\Tcal$. Let $S$ be a simple submodule of $Hom_A(T_0,X)$ such that $S=f\cdot End_A(T_0)$ for a morphism $f\colon T_0\ra X$. By $\tr{T_1}{T_0}$ we denote the $A$-module trace of $T_1$ in $T_0$. Using Lemma \ref{lemma alternative description}, it follows that $T_0/\tr{T_1}{T_0}$ lies in $\alpha(\Tcal)$ and, thus, by assumption, all maps from $T_0/\tr{T_1}{T_0}$ to $X$ are zero. Consequently, the $A$-module $\tr{T_1}{T_0}$ cannot be contained in the kernel of $f$. Therefore, there is a map $g\colon T_1\ra T_0$ such that $f\circ g\colon T_1\ra X$ is not zero. Moreover, since $T_0$ generates $\Tcal$, we also obtain a map $h\colon T_0\ra T_1$ such that the composition
$$\xymatrix{T_0\ar[r]^h & T_1\ar[r]^g & T_0\ar[r]^f & X}$$
is not trivial. But, using Lemma \ref{lem split-proj}, the map $g\circ h$ belongs to the radical of $End_A(T_0)$ and, thus, the composition $f\circ g\circ h$ lies in $S\cdot rad(End_A(T_0))=rad(S)=0$. This yields a contradiction.
\end{proof}

We summarise the results above in the following theorem.

\begin{theorem}\label{main}
Let $A$ be a finite dimensional $\mathbb{K}$-algebra. Then $\alpha$ yields a bijective correspondence between
\begin{enumerate}
\item functorially finite torsion classes in $mod(A)$;
\item functorially finite wide subcategories $\Wcal$ in $mod(A)$ for which $\Tcal_\Wcal$ is functorially finite.
\end{enumerate}
\end{theorem}

\begin{proof}
This follows from Proposition \ref{prop wide} and Proposition \ref{prop alpha injective}.
\end{proof}

Note that, in general, the assumption above of $\Tcal_\Wcal$ being functorially finite is not automatic (see Lemma \ref{lem torsion-ff} and Example \ref{example Asai}). Nevertheless, it can be dropped in some contexts.

\begin{corollary}\label{cor main}
Let $A$ be a finite dimensional $\mathbb{K}$-algebra with $\lvert f\mbox{-}tors(A) \rvert < \infty$. Then $\alpha$ yields a bijection
$$tors(A)\longrightarrow wide(A).$$
Moreover, all torsion classes and all wide subcategories in $mod(A)$ are functorially finite.
\end{corollary}

\begin{proof}
Using \cite[Theorem 3.8]{DIJ}, $\lvert f\mbox{-}tors(A)\rvert < \infty$ implies that all torsion classes in $mod(A)$ are functorially finite. Now the statement follows from Theorem \ref{main}. In fact, all wide subcategories of $mod(A)$ are functorially finite, by Proposition \ref{prop wide} and Lemma \ref{lemma alternative description}.
\end{proof}

This corollary, in particular, applies to representation finite $\mathbb{K}$-algebras. Furthermore, examples of representation infinite algebras with only finitely many torsion classes can be found in \cite{A} and \cite{Mi}, namely certain Brauer graph algebras and preprojective algebras of Dynkin type.

In some special cases, similar bijections to the ones above were discussed in the literature before.
If $A$ is a hereditary algebra, it was proven in \cite[Corollary 2.17]{IT} that $\alpha$ yields a bijection between functorially finite torsion classes and functorially finite wide subcategories in $mod(A)$. Here, $\Tcal_\Wcal$ is given by $gen(\Wcal)$. Moreover, it follows from \cite[Proposition 6.2]{M} that $\alpha$ induces a similar bijection for Nakayama algebras with $\Tcal_\Wcal$ being 
$$gen(\Wcal)\star gen(\Wcal):=\{X\in mod(A)\mid\exists\,\, 0\ra M\ra X\ra N\ra 0,\,\, M,N\in gen(\Wcal)\}.$$

We finish this section by discussing necessary and sufficient conditions for the torsion class $\Tcal_\Wcal$ to be functorially finite for $\Wcal\in f\mbox{-}wide(A)$.
Let $gen(\Wcal)^{\star n}$ be the class of modules possessing a $gen(\Wcal)$-filtration of length at most $n$ for $n\in\mathbb{N}$.

\begin{lemma} \label{lem torsion-ff}
The following are equivalent for $\Wcal$ in $f\mbox{-}wide(A)$:
\begin{enumerate}
\item $\Tcal_\Wcal$ is functorially finite.
\item There exists $n \ge 1$ such that $\Tcal_\Wcal = gen(\Wcal)^{\star n}$.
\end{enumerate}
\end{lemma}

\begin{proof}
(1) $\Rightarrow$ (2). Let $\phi: A \to T_0$ be the minimal left $\Tcal_\Wcal$-approximation of $A$. Then certainly $T_0 \in gen(\Wcal)^{\star n}$ for some $n \ge 1$, since $\Tcal_\Wcal = \bigcup_{n \ge 1} gen(\Wcal)^{\star n}$. Now given any $M \in \Tcal_\Wcal$ and a surjection $\pi: A^m \to M$, then $\pi$ factors through the left $\Tcal_\Wcal$-approximation $\phi^m\colon A^m \to T_0^m$. Thus, we have a surjection $T_0^m \to M$, proving that $M \in gen(\Wcal)^{\star n}$.

(2) $\Rightarrow$ (1). Since $\Wcal$ is functorially finite, so are $gen(\Wcal)$ (see \cite[\S4]{AS}) and $\Tcal_\Wcal = gen(\Wcal)^{\star n}$ (compare, for example, to \cite[Corollary 1.3]{GT}).
\end{proof}

% ==============================================================================
\section{Universal localisations}\label{section localisations}
Classifying functorially finite wide subcategories of $mod(A)$ at the same time translates to classifying certain ring epimorphisms  starting in $A$ (up to equivalence). Let us be a bit more precise. A ring epimorphism is an epimorphism in the category of all (unital) rings. It is well-known that ring epimorphisms describe precisely those ring homomorphisms for which restriction yields a fully faithful functor between the corresponding module categories. We say that two ring epimorphisms $f:A\ra B$ and $g:A\ra C$ are equivalent if there is a (necessarily unique) isomorphism of rings $h:B\ra C$ such that $g=h\circ f$. Moreover, for a given ring epimorphism $A\ra B$, it was proven in \cite[Theorem 4.8]{Sch} that $Tor_1^A(B,B)=0$ if and only if $Ext_A^1(M,N)\cong Ext_B^1(M,N)$ for all $B$-modules $M$ and $N$. The following proposition, essentially due to \cite{GL} (also compare to \cite[Theorem 1.6.1]{I}), is important in our context.

\begin{proposition}\label{bireflective subcat}
Let $A$ be a finite dimensional $\mathbb{K}$-algebra. By assigning to a given ring epimorphism the essential image of its associated restriction functor, one gets a bijection between
\begin{enumerate}
\item equivalence classes of ring epimorphisms $A\ra B$ with $dim_\mathbb{K}(B)<\infty\,$ and $Tor_1^A(B,B)=0$;
\item functorially finite wide subcategories of $mod(A)$.
\end{enumerate}
\end{proposition}

Using our results obtained in the previous section, we can describe many of these ring epimorphisms more explicitly by using the concept of universal localisation.

\begin{theorem}\cite[Theorem 4.1]{Sch}
Let $\Sigma$ be a set of maps between finitely generated projective $A$-modules. Then there is a ring $A_\Sigma$, the universal localisation of $A$ at $\Sigma$, and a ring homomorphism $f:A\rightarrow A_\Sigma$ such that
\begin{enumerate}
\item $A_\Sigma \otimes_A \sigma$ is an isomorphism for all $\sigma\in\Sigma$;
\item every ring homomorphism $g:A\rightarrow B$ such that $B\otimes_A \sigma$ is an isomorphism for all $\sigma\in\Sigma$ factors in a unique way through $f$, i.e. there is a commutative diagram of the form
\begin{equation}\nonumber
\xymatrix{A\ar[rr]^g\ar[rd]_{f}&&B\\ & A_\Sigma.\ar[ru]_{\exists! \tilde{g}}}
\end{equation}

\end{enumerate}
Moreover, $f\colon A\ra A_\Sigma$ is a ring epimorphism fulfilling $Tor_1^A(A_\Sigma,A_\Sigma)=0$. 
\end{theorem}

Note that, in general, the $\mathbb{K}$-dimension of $A_\Sigma$ can be infinite. If it is finite, we can consider the restriction functor $f_*:mod(A_\Sigma)\ra mod(A)$ whose essential image is given by those $X\in mod(A)$ for which $Hom_A(\sigma,X)$ is an isomorphism for all $\sigma\in\Sigma$.
The philosophy is that most (and conjecturally all) of the ring epimorphisms arising from functorially finite wide subcategories of $mod(A)$ are universal localisations. To prove this, we continue to study in more detail the objects in $f\mbox{-}tors(A)$.

Let $\Tcal=gen(T)$ be a functorially finite torsion class in $mod(A)$ with respect to the silting module $T$ and its projective presentation $\sigma$. Then $\Tcal= \Dcal_\sigma = \{X\in mod(A)\mid Hom_A(\sigma,X)\text{ surjective}\}$ (compare to Definition~\ref{defn silting}). Consider the $\Tcal$-approximation sequence
\begin{equation} \label{eq approx}
\xymatrix@1{A\ar[r]^\phi & T_0\ar[r] & T_1\ar[r] & 0} 
\end{equation}
with $T_0$ and $T_1$ in $add(T)$ and $\phi$ left-minimal. Note that $T':=T_0\oplus T_1$ is a silting module equivalent to $T$, i.e. $add(T)=add(T')$. In fact, if we choose the projective presentation $\sigma'$ of $T'$ to be the direct sum of the minimal projective presentation of $T'$ and the trivial map $Ae\ra 0$ for an idempotent $e\in A$, that is chosen maximal such that $Hom_A(Ae,T)=0$, then also $\Dcal_\sigma = \Dcal_{\sigma'}$. Similar to the situation for tilting modules one also recovers $\Tcal$ from $T_1$.

\begin{lemma}\label{lem recovering}
With the above notation, there is a projective presentation $\sigma_1$ of $T_1$ such that $\Dcal_{\sigma_1}=\Tcal$ (in particular, $T_1$ is a partial silting module with respect to $\sigma_1$ following \cite[Definition 3.10]{AMV}).
\end{lemma}

\begin{proof}
Let $\sigma'_0\in add(\sigma)$ be the minimal projective presentation of the module $T_0$ from~\eqref{eq approx}. We will view $\sigma'_0$ as a two-term complex in the derived category $D(A):=D(mod(A))$. There the approximation map $\phi: A \to T_0$ factors through the cokernel map $\pi: \sigma'_0 \to T_0$. That is, there is a map $\psi: A \to \sigma'_0$ in $D(A)$ such that $\phi = \pi \circ \psi$ and, in particular, $Hom_{D(A)}(\psi,T)$ is surjective.
Now consider the triangle
\begin{equation} \label{eq silting triangle}
\xymatrix@1{A\ar[r]^\psi & \sigma'_0\ar[r] & \sigma'_1\ar[r] & A[1].} 
\end{equation}
Then $\sigma'_1$ is also a two-term complex of finitely generated projective $A$-modules and $coker(\sigma'_1) = T_1$. A~priori, we do not know whether $\sigma'_1$ is a minimal presentation of $T_1$. However, an application of the functor $Hom_{D(A)}(-,T[1])$ yields an exact sequence
\[ \xymatrix@1{Hom_{D(A)}(\sigma'_0[1],T[1]) \ar[r]^{\psi^*} & Hom_{D(A)}(A[1],T[1]) \ar[r] & Hom_{D(A)}(\sigma'_1,T[1]) \ar[r] & Hom_{D(A)}(\sigma'_0,T[1])} \]
and, since $T \in \Dcal_\sigma$ and $Hom_{D(A)}(\psi,T)$ is surjective, it follows that $Hom_A(\sigma'_1,T)$ is surjective.
Thus, any potential summand $Ae' \to 0$ of $\sigma'_1$ is such that $Hom_A(Ae',T)=0$.

Let again $e\in A$ be a maximal idempotent such that $Hom_A(Ae,T)=0$ and let $\sigma_1$ be the map obtained from $\sigma'_1$ by adding $Ae \to 0$. 
Then we have shown that the complex $\sigma_1$ lies in $add(\sigma)$ and hence $\Dcal_{\sigma}\subseteq\Dcal_{\sigma_1}$. For the inclusion $\Dcal_{\sigma_1}\subseteq\Dcal_{\sigma}$, take a module $M$ in $\Dcal_{\sigma_1}$ and apply the functor $Hom_{D(A)}(-,M)$ to the triangle in \eqref{eq silting triangle}. We learn that $Hom_A(\sigma'_0,M)$ is surjective and since $add(\sigma) = add(\sigma'_0 \oplus \sigma_1)$ by construction, necessarily $M \in \Dcal_\sigma$.
\end{proof}

Next, we provide sufficient conditions for a ring epimorphism to be a universal localisation. 

\begin{proposition} \label{prop univ-loc-Tor}
Let $A\ra B$ be a ring epimorphism with $dim_\mathbb{K}(B)<\infty$ and $Tor_1^A(B,B)=0$. By $\Xcal_B$ we denote the essential image of the restriction functor in $mod(A)$. If $\Tcal_{\Xcal_B}=filt(gen(\Xcal_B))$ is functorially finite, then $B$ is the universal localisation of $A$ at a map $\sigma_B$ between finitely generated projective $A$-modules with $Hom_{D(A)}(\sigma_B,\sigma_B[1])=0$.
\end{proposition}

\begin{proof}
Since, by assumption, $\Tcal_{\Xcal_B}$ is functorially finite, it is of the form $gen(T)$ for some finite dimensional silting $A$-module $T$. Now consider the $gen(T)$-approximation sequence
$$\xymatrix{A\ar[r]^\phi & T_0\ar[r] & T_1\ar[r] & 0}$$
with $\phi$ left-minimal and $T_0$ and $T_1$ in $add(T)$. By Lemma \ref{lem recovering}, there is a projective presentation $\sigma_1$ of $T_1$ turning $T_1$ into a partial silting module with 
$$gen(T)=\Dcal_{\sigma_1}=\{X\in mod(A)\mid Hom_A(\sigma_1,X) \,\,\text{is surjective}\}.$$ 
We define $\sigma_B:=\sigma_1$. Using Proposition \ref{prop wide} and Lemma \ref{lemma alternative description}, it follows that
$$\Xcal_B=\alpha(gen(T))=gen(T)\cap T_1^\circ=\{X\in mod(A)\mid Hom_A(\sigma_1,X) \,\,\text{is an isomorphism}\}.$$
It remains to check that the universal localisation of $A$ at $\{\sigma_1\}$ is finite dimensional. But this is a direct consequence of the construction in \cite{AMV2}. In fact, there is an isomorphism of algebras 
$$A_{\{\sigma_1\}}\cong End_A^{op}(T_0\oplus T_1)/\langle e_{T_1} \rangle$$ 
where $e_{T_1}$ denotes the idempotent in $End_A^{op}(T_0\oplus T_1)$ associated to the summand $T_1$. 
\end{proof}

In the context of representation finite algebras, we obtain the following theorem that generalises previous work on Nakayama algebras (see \cite[Corollary 6.3]{M}).

\begin{theorem}\label{main 2}
Let $A$ be a finite dimensional and representation finite $\mathbb{K}$-algebra. Then the following sets are in bijection:
\begin{enumerate}
\item isomorphism classes of basic silting modules in $mod(A)$;
\item torsion classes in $mod(A)$;
\item wide subcategories of $mod(A)$;
\item epiclasses of ring epimorphisms $A\ra B$ with $Tor^A_1(B,B)=0$;
\item epiclasses of universal localisations of $A$. 
\end{enumerate} 
In particular, for all universal localisations $A_\Sigma$ of $A$ there is a map $\sigma$ between finitely generated projective $A$-modules with $Hom_{D(A)}(\sigma,\sigma[1])=0$ such that $A_\Sigma=A_{\{\sigma\}}$.
\end{theorem}

\begin{proof} 
The bijection $(1)\Leftrightarrow (2)$ is due to \cite[Theorem 2.7]{AIR}. Moreover, the correspondence $(2)\Leftrightarrow (3)$ follows from Corollary \ref{cor main}. The bijection $(3)\Leftrightarrow (4)$ is exactly Proposition \ref{bireflective subcat}. To see this, recall that for representation finite algebras, all ring epimorphisms $A\ra B$ have a finite dimensional codomain (\cite[Corollary 2.3]{GdP}). Finally, the bijection $(4)\Leftrightarrow (5)$ is a consequence of Proposition~\ref{prop univ-loc-Tor}.
\end{proof}

\begin{example}\label{example AIR}
Let us consider the situation from~\cite[Example 6.4]{AIR} and~\cite[Examples 5.5 and 6.6]{M}, i.e.\ the algebra $A = \Kbb Q/\langle\beta\alpha\rangle$ with the quiver
\[ Q: \xymatrix@1{1 \ar[r]^\alpha & 2 \ar[r]^\beta & 3}.\]
Then $T = S_1 \oplus P_1 \oplus P_3$ is a sincere $\tau$-tilting module which is not tilting.

Let us compute the corresponding universal localisation. We simply have $\Tcal = gen(T) = add(T)$ and the minimal $\Tcal$-approximation sequence for $P_2$ is
\[ \xymatrix@1{P_2 \ar[r]^{-\cdot \alpha} & P_1 \ar[r] & S_1 \ar[r] & 0},\]
which is at the same time the minimal projective presentation for $S_1$. Hence, the minimal approximation sequence~\eqref{eq approx} for $A$ is of the form $A \to P_1 \oplus P_1 \oplus P_3 \to S_1 \to 0$. It follows that $P_1,P_3$ are precisely the indecomposable split projective modules in $\Tcal$ while $S_1$ is the only indecomposable Ext-projective module which is not split projective. The corresponding universal localisation $f$ is obtained by inverting the minimal projective presentation of $S_1$, or in other words, we adjoin an inverse $\alpha^{-1}$ to $\alpha$ in $A$. Clearly, $f: A \to A_{\{\alpha\}}$ cannot be injective since $\beta\cdot T = 0$, and indeed one easily checks that $ker(f) = \langle\beta\rangle$. As a ring, $A_{\{\alpha\}} \cong End_A^{op}(P_1 \oplus P_1 \oplus P_3 \oplus S_1)/\langle e_{S_1} \rangle \cong M_2(\Kbb) \times \Kbb$.

Combining \cite[Example 6.4]{AIR} with Theorem~\ref{main 2}, there are precisely 12 universal localisations up to equivalence originating at our $A$. Similar easy computations can be performed for a range of other representation finite algebras whose representation theory is sufficiently well understood; see e.g.~\cite[\S6]{AIR} or~\cite{M}.
\end{example}

The following example was communicated to us by Sota Asai, see~\cite[Example 4.5]{Asai}.
\begin{example}\label{example Asai}
Consider the algebra $A = \Kbb Q/\langle\gamma\alpha\rangle$ with the quiver
\[ Q: \xymatrix@1{1 \ar@<.5ex>[r]^\alpha\ar@<-.5ex>[r]_\beta  & 2 \ar[r]^\gamma & 3}.\]
Note that $A$ is a tilted algebra of type $\widetilde{A}_2$. Since the indecomposable injective $A$-module $I_3$ has no non-trivial endomorphisms, it follows that $\Wcal=add(I_3)$ is a functorially finite wide subcategory of $mod(A)$. It was shown by Asai that the torsion class $\Tcal_\Wcal$ is not functorially finite. Indeed, otherwise also the intersection $\Tcal_\Wcal\cap P_3^\circ=filt(gen(I_3/S_3))$ would be functorially finite in $\Xcal_{A/Ae_3A}\cong mod(A/Ae_3A)$, yielding a contradiction, since $I_3/S_3$ identifies with a simple regular module over the Kronecker algebra $A/Ae_3A$. In particular, it follows that the assignment $\alpha$ from Theorem \ref{main} only yields an injective map $f\mbox{-}tors(A)\rightarrow f\mbox{-}wide(A)$ and not a bijection (see~\cite[Theorem 4.6]{Asai} for when $f\mbox{-}tors(A)\rightarrow f\mbox{-}wide(A)$ is a bijection for other tilted algebras $A$).

Nevertheless, it turns out that the functorially finite wide subcategory $\Wcal$ coincides with the essential image of the restriction functor of the universal localisation of $A$ at the map
$$\sigma:=(\xymatrix@1{P_3 \ar[r]^{-\cdot \gamma} & P_2})\oplus (\xymatrix@1{P_3 \ar[r]^{-\cdot \gamma\beta} & P_1}).$$
Indeed, on the one hand, $I_3$ clearly carries an $A_{\{\sigma\}}$-module structure and, on the other hand, by adjoining inverses $\gamma^{-1}$ and $(\gamma\beta)^{-1}$ to $\gamma$ and $\gamma\beta$ in $A$ we obtain the localisation $A_{\{\sigma\}}$ which is Morita equivalent to $\Kbb$. Hence, we have $\Wcal=\Xcal_{A_{\{\sigma\}}}$. Moreover, it is a straightforward computation to check that $Hom_{D(A)}(\sigma,\sigma[1])=0$. It follows that not all universal localisations of $A$ at such a map $\sigma$ arise from the process described in Proposition \ref{prop univ-loc-Tor} and Example \ref{example AIR}. In other words, different techniques are needed if one hopes to prove a general version of Theorem \ref{main 2} relating silting modules to universal localisations.
\end{example}

% ==============================================================================

\end{document}